\documentclass[11pt]{article}

\usepackage[a4paper,margin=1in]{geometry}
\usepackage{amsmath,amssymb,amsthm,mathtools}
\usepackage{booktabs}
\usepackage{enumitem}
\usepackage{microtype}
\usepackage[hidelinks]{hyperref}
\hypersetup{
  pdftitle={Cyclic Kronecker Fourier Transforms, Weber Norms, and Rogers--Ramanujan Singular Values},
  pdfauthor={Sumit Kumar Jha},
  pdfsubject={Kronecker limit formulas, generalized Weber functions, class fields, and Rogers--Ramanujan singular values},
  pdfkeywords={Rogers--Ramanujan continued fraction, Kronecker limit formula, Dedekind eta function, generalized Weber function, class field theory, Hecke L-value}
}

\newtheorem{theorem}{Theorem}[section]
\newtheorem{proposition}[theorem]{Proposition}
\newtheorem{lemma}[theorem]{Lemma}
\newtheorem{corollary}[theorem]{Corollary}
\theoremstyle{remark}
\newtheorem{remark}[theorem]{Remark}

\newcommand{\Cl}{\operatorname{Cl}}
\newcommand{\Gal}{\operatorname{Gal}}
\newcommand{\RR}{\mathbb{R}}
\newcommand{\QQ}{\mathbb{Q}}
\newcommand{\CC}{\mathbb{C}}

\title{Cyclic Kronecker Fourier Transforms, Weber Norms,\\
and Rogers--Ramanujan Singular Values}
\author{Sumit Kumar Jha\\
\small Greater Noida, Uttar Pradesh, India\\
\small \texttt{sumitkumarjha75@gmail.com}}
\date{}

\begin{document}
\maketitle

\begin{abstract}
Let $K$ be an imaginary quadratic field with fundamental discriminant
$D=-5n$, where $n>0$ is odd, and cyclic ideal class group
$\Cl(K)=\langle g\rangle\simeq C_{4M}$.  Assume that the ramified
prime class above $5$ is the unique involution $g^{2M}$.  For the
Kronecker eta invariant
\[
 F(A)=\frac{|\eta(z_A)|^2}{\sqrt{a_A}},
\]
where $z_A$ and $a_A$ arise from a positive definite quadratic form
representing the ideal class $A$, we show that the logarithms of the mirror ratios
\[
 R_j=\frac{F(g^{2M-j})}{F(g^j)},\qquad 0\le j<M,
\]
are related by an explicit discrete cosine transform to the odd
class-character Hecke $L(1)$-values.  A single level-$5$
generalized-Weber orbit realizes all of these mirrors.  Combining Fourier
nonvanishing, lower-power Weber class invariants and a ray-class cancellation
argument, we prove
\[
 \QQ(R_j)=H^{\langle c\rangle},\qquad
 R_j^{\sigma^{2M}}=R_j^{-1},\qquad
 \QQ(R_j+R_j^{-1})=H^{\langle c,\sigma^{2M}\rangle}.
\]
Here $H$ is the Hilbert class field, $c$ denotes complex conjugation,
and $\sigma$ corresponds to $g$ under the Artin map.  Consequently
\[
 [\QQ(R_j):\QQ]=4M,\qquad
 [\QQ(R_j+R_j^{-1}):\QQ]=2M,
\]
and the latter degree becomes $M$ over the real genus field
$\QQ(\sqrt5)$.  Thus the degree pattern $2^r,2^{r-1},2^{r-2}$ in the cyclic
$2$-power hierarchy is a specialization of a uniform $C_{4M}$ theorem.
We give exact low-class-number evaluations and certified computations at
$D=-895,-3295,-8495$, separating respectively four, eight and sixteen
non-genus Hecke $L(1)$-values and producing corresponding
alternating Rogers--Ramanujan singular values.  Classical mixed modular
equations then give ordinary Rogers--Ramanujan companions at the same
arguments.
\end{abstract}



\section{Introduction}

The Rogers--Ramanujan continued fraction
\[
R(q)=\cfrac{q^{1/5}}{1+\cfrac{q}{1+\cfrac{q^2}{1+\cdots}}}
\]
and its alternating companion
\[
S(q)=-R(-q)
\]
have a long history of singular-value evaluations at quadratic
arguments.  
\paragraph{Conventions.}
For $\tau$ in the upper half-plane, let $q=e^{2\pi i\tau}$ and
\[
\eta(\tau)=q^{1/24}\prod_{m\ge1}(1-q^m)
\]
be Dedekind's eta function.  We take the positive real fifth root when
$0<q<1$; in $R(-q)$ the fifth root of a negative real number is chosen
real.  Write $\QQ$, $\RR$, and $\CC$ for the rational, real, and complex
numbers, respectively.  The notation $C_m$ denotes a cyclic group of order
$m$.  For an imaginary quadratic field $K$ of fundamental discriminant
$D$, write $\Cl(K)$ for its ideal class group and
$h(D)=|\Cl(K)|$ for its class number.  If $H/K$ is the Hilbert class
field and $J\subseteq\Gal(H/\QQ)$ is a subgroup, $H^J$ denotes the
subfield fixed by $J$.  For $x\ge1$ we use
$\operatorname{arcosh}x=\log(x+\sqrt{x^2-1})$.

Ramanathan developed a method based on eta quotients,
ideal classes and Kronecker's limit formula~\cite{Ramanathan}.
Berndt, Chan and Zhang combined Kronecker-limit calculations, modular
equations and class field theory to evaluate further Ramanujan class
invariants~\cite{BCZ}.  For an account of the modular-function background, see
Duke~\cite{Duke}; recent icosahedral evaluations of ordinary
Rogers--Ramanujan singular values are given by Campbell~\cite{Campbell}.

In earlier work of the present author~\cite{JhaFour}, the one-class-per-genus case was
treated by finite Fourier inversion on an elementary $2$-class group.
When every genus consists of one class, every class-group character is
quadratic, and the relevant class-group $L$-values factor into products
of quadratic Dirichlet $L$-values.  The resulting eta-class ratios are
therefore explicit products of real quadratic fundamental units.

The purpose of the present paper is to study the complementary cyclic
setting
\[
\Cl(K)\simeq C_{4M},
\]
with the $2$-power hierarchy $C_4,C_8,C_{16},\ldots$ providing the cleanest
explicit tower.  Higher-order class-group characters now occur, and the
genus characters alone no longer isolate individual eta-class invariants.
Nevertheless the full Kronecker transform remains elementary finite Fourier
analysis.  The main observation is that the real symmetry
\[
F(A)=F(A^{-1})
\]
converts the odd character modes into a discrete cosine transform of a
small collection of mirrored eta-class ratios.

The resulting transform gives a natural hierarchy
\[
C_4,\quad C_8,\quad C_{16},\quad C_{32},\quad \ldots,
\]
in which one needs respectively $1,2,4,\ldots$ mirror ratios to recover
all odd non-genus modes.  The norm-$5$ mirror is exactly the quantity
appearing in Ramanathan's Rogers--Ramanujan eta quotient.  More strongly,
when the ramified prime above $5$ is the unique order-$2$ class, every mirror
is a norm of a singular value of the same level-$5$ generalized Weber
function.  Thus the entire cosine transform sits in one CM Weber orbit.
The explicit computations below emphasize the $2$-power tower, while the
class-field descent is proved uniformly for the full $C_{4M}$ setting.

\paragraph{Relation to earlier work and the new point.}
The ingredients entering this picture have substantial precedents, but they
have usually been used for different purposes.  Berndt--Chan--Zhang used
Kronecker's limit formula together with modular equations to evaluate
Ramanujan class invariants~\cite{BCZ}, while Gee--Honsbeek studied singular
values of the Rogers--Ramanujan continued fraction by complex multiplication
and Shimura reciprocity~\cite{GeeHonsbeek}.  Gee and, systematically,
Enge--Morain developed the class-invariant and Galois theory of modular
functions and generalized Weber functions~\cite{Gee,EngeMorain}; eta-quotient
generators of ring class fields have also been studied directly, for example
in~\cite{KooShinYoon}.  Finally, the passage between partial zeta data and
class-character $L$-values is of course finite Fourier analysis on the class
group.

The contribution here is therefore not the abstract existence of a Fourier
transform or of Weber class invariants.  It is the specific synthesis forced
by the ramified norm-$5$ involution: pairing $g^j$ with $g^{2M-j}$ kills the
even character modes and turns the odd part of the Kronecker transform into
an explicit cosine transform; the entire mirror vector is then realized by
one level-$5$ Weber orbit.  The ray-class descent further shows, for every
mirror in the standing $C_{4M}$ family, that the \emph{unpowered} positive
ratio itself generates $H^{\langle c\rangle}$ and that its reciprocal trace
generates $H^{\langle c,\sigma^{2M}\rangle}$.  The following theorems establish the mirror transform and the associated
field-generation statements.

\paragraph{Classical benchmarks and headline evaluations.}
We recall two classical evaluations of Ramanujan's continued fraction
for comparison.  With
\[
 \phi=\frac{1+\sqrt5}{2},
\]
one has~\cite{BCZ96}
\begin{equation}
 R(e^{-2\pi})=\sqrt{\frac{5+\sqrt5}{2}}-\phi,
 \qquad
 S(e^{-\pi})=\sqrt{\frac{5-\sqrt5}{2}}-(\phi-1).
\label{eq:classical-RS-benchmarks}
\end{equation}
A second classical value, especially suggestive for the norm-$5$ theme, is
\begin{equation}
 R(e^{-2\pi\sqrt5})
 =\frac{\sqrt5}
 {1+\left(5^{3/4}(\phi-1)^{5/2}-1\right)^{1/5}}-\phi.
\label{eq:classical-R-sqrt5}
\end{equation}
These formulas are recalled only as benchmarks; the new point here is that
cyclic class-group Fourier inversion produces similarly explicit
\emph{non-genus} $L$-values together with alternating Rogers--Ramanujan
singular values.

For $D=-155$ (class group $C_4$), define
\[
C_{31}
=
S(e^{-\pi/\sqrt{155}})^{-5}
+11
-
S(e^{-\pi/\sqrt{155}})^5.
\]
Then
\begin{equation}
C_{31}
=
625+280\sqrt5
-
50\sqrt{313+140\sqrt5}
\label{eq:intro-C155}
\end{equation}
and hence
\begin{equation}
S(e^{-\pi/\sqrt{155}})
=
\left(
\frac{11-C_{31}+\sqrt{(11-C_{31})^2+4}}{2}
\right)^{1/5}.
\label{eq:intro-S155}
\end{equation}
The associated order-$4$ class-character value has the particularly compact
form
\begin{equation}
L_{\QQ(\sqrt{-155})}(1,\chi)
=
\frac{2\pi}{3\sqrt{155}}
\operatorname{arcosh}(56+25\sqrt5).
\label{eq:intro-L155}
\end{equation}
For $D=-95$ (class group $C_8$), put
\[
 A_{19}=\frac{21+10\sqrt5+5\sqrt{31+14\sqrt5}}2,
 \qquad
 B_{19}=\frac{\sqrt5+\sqrt{2\sqrt5-1}}2.
\]
The positive branch of the eta quotient gives
\[
 C_{19}
 =\frac{5\sqrt5}{2}
 \left(A_{19}-\sqrt{A_{19}^2-4}\right),
\]
and the corresponding alternating value is
\begin{equation}
 S(e^{-\pi/\sqrt{95}})
 =\left(
 \frac{11-C_{19}+\sqrt{(11-C_{19})^2+4}}2
 \right)^{1/5}.
\label{eq:intro-S95}
\end{equation}
Choose a generator $g$ of the class group $C_8$, and for $k=1,3$ let
$\chi_k(g)=e^{2\pi i k/8}$.  The non-genus character values satisfy
\begin{equation}
L_K(1,\chi_k)
=\frac{2\pi}{\sqrt{95}}
\left
\{
 \frac13\operatorname{arcosh}\!\frac{A_{19}}2
 \ \pm\ \sqrt2\,\operatorname{arcosh}\!\frac{B_{19}}2
\right\},
\label{eq:intro-L95-pair}
\end{equation}
where the plus sign corresponds to $k=1$ and the minus sign to $k=3$.

At the next level, $D=-895$ has cyclic class group $C_{16}$.  One of
the four mirror ratios is
\[
R_3=\frac{F(g^5)}{F(g^3)}.
\]
Writing $Y_3=R_3+R_3^{-1}$, one of the four certified trace equations is
the compact quartic
\begin{equation}
Y_3^4
+\frac{11-\sqrt5}{2}Y_3^3
+\frac{95-43\sqrt5}{2}Y_3^2
+\frac{-181+59\sqrt5}{2}Y_3
+\frac{123-55\sqrt5}{2}
=0.
\label{eq:intro-D895-preview}
\end{equation}
The cases $D=-155$ and $D=-95$ are evaluated in
Sections~\ref{sec:C4} and~\ref{sec:C8}.  Section~\ref{sec:D895-certificate}
treats all four mirror ratios at $D=-895$ using one level-$5$ system.

\section{Kronecker eta invariants and the full class-group transform}

Let $D<0$ be a fundamental discriminant and
\[
K=\QQ(\sqrt D).
\]
A primitive positive definite binary quadratic form $[a,b,c]$ has
discriminant $b^2-4ac=D$.  Its proper equivalence class represents an
ideal class $A\in\Cl(K)$.  For such a representative, set
\[
z_A=\frac{-b+\sqrt D}{2a},
\qquad
F(A)=\frac{|\eta(z_A)|^2}{\sqrt a},
\]
where $\sqrt D=i\sqrt{|D|}$ has positive imaginary part.  The
transformation law of $\eta$ makes $F(A)$ independent of the
representative.  Conjugation of ideal classes gives $F(A)=F(A^{-1})$.

For a character $\chi$ of $\Cl(K)$, let $L_K(s,\chi)$ denote the
unramified Hecke $L$-function determined in $\Re(s)>1$ by
\[
L_K(s,\chi)=\sum_{\mathfrak a\subseteq\mathcal O_K,\,\mathfrak a\ne0}
\frac{\chi([\mathfrak a])}{(\mathrm N\mathfrak a)^s},
\]
where the sum runs over nonzero integral ideals of the ring of integers
$\mathcal O_K$ and $\mathrm N\mathfrak a$ is the absolute norm.
For a nonprincipal class-group character $\chi$, put
\[
\Lambda_\chi
=
\frac{\sqrt{|D|}}{2\pi}L_K(1,\chi).
\]
For every nonprincipal class-group character, the character form of
Kronecker's limit formula gives~\cite{BCZ}
\begin{equation}
\Lambda_\chi
=
-\sum_{A\in\Cl(K)}\chi(A)\log F(A).
\label{eq:KLF}
\end{equation}

\begin{proposition}[Full Fourier extraction]
Let $A_0$ be the principal class and let $C$ be any class.  Then
\[
\log\frac{F(A_0)}{F(C)}
=
-\frac1{h(D)}
\sum_{\chi\neq 1}
\left(1-\overline{\chi(C)}\right)\Lambda_\chi.
\]
\end{proposition}

\begin{proof}
Here $\delta_{A,B}$ denotes the Kronecker delta, equal to $1$ when
$A=B$ and to $0$ otherwise.  Apply character
orthogonality to
\[
\delta_{A,A_0}-\delta_{A,C}
=
\frac1{h(D)}
\sum_{\chi}
\left(1-\overline{\chi(C)}\right)\chi(A),
\]
multiply by $\log F(A)$ and sum over $A$.  The principal character
drops out and \eqref{eq:KLF} gives the result.
\end{proof}

When $C$ has order $2$, every character takes the value $\pm1$ on $C$,
and the formula simplifies to
\[
\log\frac{F(A_0)}{F(C)}
=
-\frac{2}{h(D)}
\sum_{\chi(C)=-1}\Lambda_\chi.
\]

\section{The cyclic Kronecker cosine transform}

Assume now
\[
G=\Cl(K)=\langle g\rangle\simeq C_{4M},
\qquad M\ge1,
\]
and write
\[
f_j=\log F(g^j),\qquad 0\le j<4M.
\]
Since $F(A)=F(A^{-1})$,
\[
f_j=f_{4M-j}.
\]

For odd $k$ with $1\le k<2M$, let
\[
\chi_k(g)=e^{2\pi i k/(4M)}
\]
and write $\Lambda_k=\Lambda_{\chi_k}$.  Since $F(A)=F(A^{-1})$,
pairing $A$ with $A^{-1}$ in \eqref{eq:KLF} gives
\[
 \Lambda_{\bar\chi}=\Lambda_\chi\in\RR.
\]
In particular $\Lambda_{4M-k}=\Lambda_k$, a fact used in the cosine
reduction below.

Define the mirror differences
\[
d_j=f_{2M-j}-f_j,
\qquad 0\le j<M.
\]

\begin{lemma}[Odd cosine orthogonality]
Let $\delta_{k,\ell}$ be $1$ when $k=\ell$ and $0$ otherwise.
For odd integers $k,\ell$ with $1\le k,\ell<2M$,
\[
1+2\sum_{j=1}^{M-1}
\cos\frac{k\pi j}{2M}
\cos\frac{\ell\pi j}{2M}
=
M\delta_{k,\ell}.
\]
Equivalently, the matrix
\[
\left(\cos\frac{k\pi j}{2M}\right)_
{\substack{0\le j<M\\1\le k<2M,\ k\ {\rm odd}}}
\]
is invertible with the normalization used below.
\end{lemma}

\begin{proof}
Use $2\cos x\cos y=\cos(x-y)+\cos(x+y)$ and sum the resulting
finite geometric series.  Since $k$ and $\ell$ are odd, the possible
frequencies are even multiples of $\pi/(2M)$; the sums vanish unless
$k=\ell$.
\end{proof}

\begin{theorem}[Cyclic Kronecker cosine transform]\label{thm:cyclic-cosine}
For $0\le j<M$,
\begin{equation}
d_j=
\frac1M
\sum_{\substack{1\le k<2M\\ k\ {\rm odd}}}
\Lambda_k
\cos\frac{k\pi j}{2M}.
\label{eq:forwardDCT}
\end{equation}
Conversely, for every odd $k$ with $1\le k<2M$,
\begin{equation}
\Lambda_k=
d_0+
2\sum_{j=1}^{M-1}
d_j\cos\frac{k\pi j}{2M}.
\label{eq:inverseDCT}
\end{equation}
\end{theorem}

\begin{proof}
With the convention in \eqref{eq:KLF}, Fourier inversion gives
\[
f_t-f_s
=
-\frac1{4M}
\sum_{\ell\ne0}
\Lambda_\ell
\left(
e^{-2\pi i\ell t/(4M)}
-
e^{-2\pi i\ell s/(4M)}
\right);
\]
the principal-character term cancels in the difference.
Take $t=2M-j$ and $s=j$.  Since
\[
e^{-2\pi i\ell(2M-j)/(4M)}
=(-1)^\ell e^{2\pi i\ell j/(4M)},
\]
the terms with even $\ell$ cancel after pairing $\ell$ with
$4M-\ell$.  For odd $\ell$, the paired contribution is a cosine.
Using $\Lambda_\ell=\Lambda_{4M-\ell}$ gives
\[
d_j=
\frac1M
\sum_{\substack{1\le k<2M\\k\ {\rm odd}}}
\Lambda_k\cos\frac{k\pi j}{2M},
\]
which is \eqref{eq:forwardDCT}.  The preceding orthogonality lemma
gives \eqref{eq:inverseDCT}.
\end{proof}

\begin{corollary}[Dirichlet-kernel form]
Put
\[
e_m=d_m-d_{m+1},
\qquad 0\le m<M,
\qquad d_M=0.
\]
Then
\[
\Lambda_k=
\sum_{m=0}^{M-1}
\frac{
\sin\!\left((2m+1)k\pi/(4M)\right)
}{
\sin\!\left(k\pi/(4M)\right)
}
\,e_m.
\]
\end{corollary}

\begin{proof}
Since
\[
d_j=\sum_{m=j}^{M-1}e_m,
\]
substitute into \eqref{eq:inverseDCT}.  The bracket
\[
1+2\sum_{j=1}^{m}\cos(j\theta)
\]
is the Dirichlet kernel
\[
\frac{\sin((2m+1)\theta/2)}{\sin(\theta/2)}.
\]
\end{proof}

\begin{theorem}[Main arithmetic theorem]\label{thm:main}\label{thm:unpowered-descent-v14}
Let $D<0$ be a fundamental discriminant with $5\mid D$ and $D/5$ odd, and
suppose
\[
 \Cl(K)=\langle g\rangle\simeq C_{4M},\qquad M\ge1.
\]
Assume that the ramified prime class $\mathfrak p$ above $5$ is the unique
involution $g^{2M}$.  For $0\le j<M$ put
\[
 R_j=\frac{F(g^{2M-j})}{F(g^j)},\qquad
 Y_j=R_j+R_j^{-1}.
\]
Let $H$ be the Hilbert class field, write
$\Gal(H/K)=\langle\sigma\rangle$ with $\sigma$ corresponding to $g$, and
let $c$ be complex conjugation.  Then
\begin{equation}\label{eq:unpowered-generator-v14}
 \QQ(R_j)=H^{\langle c\rangle},
\end{equation}
\begin{equation}\label{eq:R-halfturn-v14}
 R_j^{\sigma^{2M}}=R_j^{-1},
\end{equation}
and
\begin{equation}\label{eq:unpowered-trace-v14}
 \QQ(Y_j)=H^{\langle c,\sigma^{2M}\rangle}.
\end{equation}
Consequently
\begin{equation}\label{eq:exact-degrees-v14}
 [\QQ(R_j):\QQ]=4M,\qquad
 [\QQ(Y_j):\QQ]=2M,\qquad
 [\QQ(\sqrt5,Y_j):\QQ(\sqrt5)]=M.
\end{equation}
When $4M=2^r$, these are $2^r,2^{r-1},2^{r-2}$.
\end{theorem}

The proof occupies the remainder of Section~\ref{sec:norm5-weber}: the
cosine transform supplies primitivity, generalized Weber functions place a
lower power in $H$, and a ray-class cancellation removes the final square
root.

\section{The Rogers--Ramanujan norm-5 mirror and proof of the main theorem}
\label{sec:norm5-weber}

\begin{lemma}[Shape of the discriminant]\label{lem:shape-D}
Under the hypotheses of Theorem~\ref{thm:main}, there is a prime
$\ell\equiv3\pmod4$, $\ell\ne3$, such that
\[
 D=-5\ell.
\]
In particular $3\nmid D$, and the real genus field contained in $H$ is
$\QQ(\sqrt5)$.
\end{lemma}

\begin{proof}
Since the cyclic group $C_{4M}$ has $2$-rank one, genus theory gives exactly
two prime-discriminant factors in $D$.  One is $5$.  Because $D/5$ is odd,
the other is a negative odd prime discriminant $-\ell$, with
$\ell\equiv3\pmod4$.  If $\ell=3$, then $D=-15$ and $h(D)=2$, contrary to
$h(D)=4M\ge4$.  The assertion about the real genus field follows from the
two prime-discriminant factors $5$ and $-\ell$.
\end{proof}

Write $D=-5n$ with $n=\ell$ as in Lemma~\ref{lem:shape-D}.  The ramified
norm-$5$ class is
\[
\mathfrak p=g^{2M}.
\]

Put
\[
X_n=S(e^{-\pi/\sqrt{5n}})
\]
and
\[
C_n=X_n^{-5}+11-X_n^5.
\]
Ramanathan's eta transformation formula, in the normalization used
in~\cite{Ramanathan,JhaFour}, gives
\[
C_n=
5\sqrt5
\left(
\frac{F(1)}{F(\mathfrak p)}
\right)^3.
\]
Since
\[
d_0=\log\frac{F(\mathfrak p)}{F(1)},
\]
we obtain

\begin{proposition}
With the preceding hypotheses,
\[
C_n=5\sqrt5\,e^{-3d_0}.
\]
Equivalently,
\[
C_n
=
5\sqrt5
\exp\left(
-\frac{3}{M}
\sum_{\substack{1\le k<2M\\ k\ {\rm odd}}}
\Lambda_k
\right).
\]
\end{proposition}

Thus the Rogers--Ramanujan singular value sees the average of all odd
character modes.

\subsection*{The whole mirror vector as one level-5 Weber orbit}

The same ramified prime that produces the Rogers--Ramanujan mirror actually
realizes every mirror in the cyclic cosine transform.

\begin{theorem}[Level-$5$ realization of the cyclic mirrors]
\label{thm:all-mirrors-w5}
Assume
\[
\Cl(K)=\langle g\rangle\simeq C_{4M},
\qquad
\mathfrak p=g^{2M},
\]
where $\mathfrak p$ is the ramified prime class above $5$.  Choose a
$5$-system of primitive forms
\[
[A_j,B_j,C_j],
\qquad 0\le j<4M,
\]
ordered by the classes $g^j$, with
\[
(A_j,5)=1,
\qquad
5\mid C_j,
\qquad
B_j\equiv5\pmod{10},
\]
and put
\[
\tau_j=\frac{-B_j+\sqrt D}{2A_j}.
\]
Then for $0\le j<M$,
\begin{equation}
\frac{F(g^{2M-j})}{F(g^j)}
=\frac1{\sqrt5}
\left|\frac{\eta(\tau_j/5)}{\eta(\tau_j)}\right|^2.
\label{eq:all-mirrors-w5}
\end{equation}
Equivalently, every mirror difference is
\[
d_j=
2\log|\mathfrak w_5(\tau_j)|-\frac12\log5,
\qquad
\mathfrak w_5(\tau)=\frac{\eta(\tau/5)}{\eta(\tau)}.
\]
Moreover $\mathfrak w_5(\tau_j)^6$ is a ring-class invariant, so all mirror
ratios are algebraic.
\end{theorem}

\begin{proof}
The point $\tau_j/5$ is represented by
\[
[5A_j,B_j,C_j/5].
\]
In the form class group this is the product of $g^j$ with the ramified
norm-$5$ class $\mathfrak p=g^{2M}$, hence it represents $g^{j+2M}$.
Therefore
\[
\frac{F(g^{j+2M})}{F(g^j)}
=\frac1{\sqrt5}
\left|\frac{\eta(\tau_j/5)}{\eta(\tau_j)}\right|^2.
\]
Since $F(A)=F(A^{-1})$ and
$-(j+2M)\equiv2M-j\pmod{4M}$, the left side is exactly the mirror ratio in
\eqref{eq:all-mirrors-w5}.  Finally, for $N=5$ the canonical exponent in
Enge--Morain's generalized-Weber theorem is $6$; their $N$-system theorem
therefore places $Z_j=\mathfrak w_5(\tau_j)^6$ in the ring class field and
gives its class-group conjugates~\cite[Theorem~4.3]{EngeMorain}.  Since
\[
R_j^6=\frac{Z_j\overline{Z_j}}{125},
\]
the mirror ratio is algebraic as claimed.
\end{proof}

\begin{remark}
The Rogers--Ramanujan quantity corresponds only to $j=0$, but
Theorem~\ref{thm:all-mirrors-w5} shows that the remaining non-genus modes do
not require unrelated modular functions: they are other points in the same
level-$5$ generalized-Weber orbit.  Thus the cyclic discrete cosine transform
and Ramanathan's norm-$5$ construction are two aspects of a single CM
mechanism.
\end{remark}

\subsection*{Primitive powered mirrors from Fourier nonvanishing}

The degree pattern in the examples can be pushed considerably further without
any additional class-polynomial computation.  The useful algebraic object is
first the powered mirror $R_j^6$, whose membership in the Hilbert class field
follows directly from the canonical level-$5$ Weber invariant.  A second,
classical Kronecker invariant supplies the missing primitivity.

Let $g_{\rm MM}(A)$ denote the invariant normalized by Murty--Murty
\cite{MurtyMurty}, and rescale it by
\[
 \mathfrak G(A):=(2\pi)^{12}g_{\rm MM}(A).
\]
This harmless global rescaling leaves every quotient and Artin action
unchanged.  If $A$ is represented by $[a,b,c]$, with
$z_A=(-b+\sqrt D)/(2a)$, their definition then gives
\begin{equation}\label{eq:G-F-v12}
 |\mathfrak G(A)|=F(A)^{12}.
\end{equation}
Indeed, for the ideal
$\langle a,(-b+\sqrt D)/2\rangle$ one has norm $a$ and
$\Delta(z)=(2\pi)^{12}\eta(z)^{24}$, while Murty--Murty's invariant carries
an overall factor $(2\pi)^{-12}$; after the displayed rescaling the remaining
factors collapse to $a^{-6}\eta(z_A)^{24}$.

Two standard CM facts will be used.  First,
$\mathfrak G(A)/\mathfrak G(B)$ lies in the Hilbert class field $H$.
Second, after choosing the Artin generator $\sigma$ compatibly with $g$, one
has
\begin{equation}\label{eq:G-Artin-v12}
 \left(\frac{\mathfrak G(A)}{\mathfrak G(B)}\right)^{\sigma^i}
 =\frac{\mathfrak G(g^iA)}{\mathfrak G(g^iB)}.
\end{equation}
These are the classical transformation statements recorded explicitly in
\cite[Props.~13--14]{MurtyMurty}.

\begin{theorem}[Primitive powered mirrors]
\label{thm:primitive-powered-mirror-v12}
Keep the hypotheses of Theorem~\ref{thm:all-mirrors-w5}, put
$h=4M$, and let $H$ be the Hilbert class field of $K$.  Write
\[
 \Gal(H/K)=\langle\sigma\rangle\simeq C_h
\]
and let $c$ be complex conjugation.  For $0\le j<M$, put
\[
 R_j=\frac{F(g^{2M-j})}{F(g^j)}
     =\frac{F(g^{j+2M})}{F(g^j)}.
\]
Then
\begin{equation}\label{eq:powered-generator-v12}
 \QQ(R_j^6)=H^{\langle c\rangle}.
\end{equation}
In particular
\[
 [\QQ(R_j^6):\QQ]=h.
\]
Moreover, if
\[
 B_j:=R_j^{24},\qquad T_j:=B_j+B_j^{-1},
\]
then
\begin{equation}\label{eq:powered-trace-field-v12}
 \QQ(T_j)=H^{\langle c,\sigma^{2M}\rangle},
 \qquad [\QQ(T_j):\QQ]=\frac h2.
\end{equation}
The field $\QQ(T_j)$ contains the real genus field $\QQ(\sqrt5)$, and hence
\begin{equation}\label{eq:powered-genus-degree-v12}
 [\QQ(\sqrt5,T_j):\QQ(\sqrt5)]=\frac h4.
\end{equation}
\end{theorem}

\begin{proof}
Set
\[
 f_t=\log F(g^t),\qquad
 w_t=f_{t+2M}-f_t\qquad(t\bmod h).
\]
Since $F(A)=F(A^{-1})$, the original mirror is $R_j=e^{w_j}$.
Let
\[
 \zeta=e^{2\pi i/h},\qquad
 \chi_k(g)=\zeta^k,
\]
and retain the notation
\[
 \Lambda_k=\frac{\sqrt{|D|}}{2\pi}L_K(1,\chi_k).
\]
The Kronecker limit formula gives
\[
 \Lambda_k=-\sum_{t=0}^{h-1}\zeta^{kt}f_t.
\]
Consequently
\begin{equation}\label{eq:w-Fourier-v12}
 \sum_{t=0}^{h-1}w_t\zeta^{kt}
 =\bigl(1-(-1)^k\bigr)\Lambda_k.
\end{equation}
In particular the primitive frequency is
\[
 \sum_t w_t\zeta^t=2\Lambda_1\ne0,
\]
because the Hecke $L$-value of a nonprincipal class-group character is
nonzero at $s=1$; see, for example, \cite[Ch.~V, Thm.~2.4(b)]{MilneCFT}.

Now define
\[
 Q_j:=\frac{\mathfrak G(g^{j+2M})}{\mathfrak G(g^j)}\in H,
 \qquad
 B_j^{\ast}:=Q_j\overline{Q_j}\in H^{\langle c\rangle}.
\]
By \eqref{eq:G-F-v12},
\[
 B_j^{\ast}=R_j^{24}=B_j.
\]
Using \eqref{eq:G-Artin-v12},
\[
 \log|Q_j^{\sigma^i}|=12w_{j+i}.
\]
Since $c\sigma^i=\sigma^{-i}c$, it follows that
\begin{equation}\label{eq:B-log-orbit-v12}
 \ell_i:=\log|B_j^{\sigma^i}|
 =12\bigl(w_{j+i}+w_{j-i}\bigr).
\end{equation}
Taking the first finite Fourier coefficient and using
\eqref{eq:w-Fourier-v12} together with
$\Lambda_{-1}=\Lambda_1$ gives
\begin{equation}\label{eq:B-first-mode-v12}
 \sum_{i=0}^{h-1}\ell_i\zeta^i
 =48\Lambda_1\cos\frac{2\pi j}{h}.
\end{equation}
For $0\le j<M$ the cosine is strictly positive, so this coefficient is
nonzero.

Suppose now that $B_j^{\sigma^s}=B_j$.  Applying $\sigma^i$ shows that the
sequence $\ell_i$ is periodic with shift $s$.  A nonzero first Fourier
coefficient can be invariant under that shift only if $\zeta^s=1$.
Thus $s\equiv0\pmod h$: no nontrivial rotation fixes $B_j$.  Since $c$ fixes
$B_j$, and a reflection $c\sigma^s$ fixes it only when the rotation
$\sigma^s$ does, one obtains
\[
 \operatorname{Stab}_{\Gal(H/\QQ)}(B_j)=\langle c\rangle.
\]
Hence
\begin{equation}\label{eq:B-generates-real-v12}
 \QQ(B_j)=H^{\langle c\rangle}.
\end{equation}

Theorem~\ref{thm:all-mirrors-w5} gives
$Z_j=\mathfrak w_5(\tau_j)^6\in H$, and therefore
\[
 R_j^6=\frac{Z_j\overline{Z_j}}{125}\in H^{\langle c\rangle}.
\]
Since $B_j=(R_j^6)^4$, \eqref{eq:B-generates-real-v12} yields the chain
\[
 H^{\langle c\rangle}=\QQ(B_j)
 \subseteq\QQ(R_j^6)
 \subseteq H^{\langle c\rangle},
\]
which proves \eqref{eq:powered-generator-v12}.

It remains to analyze the reciprocal trace.  From
\eqref{eq:G-Artin-v12}, the half-turn satisfies
\[
 Q_j^{\sigma^{2M}}=Q_j^{-1},
\]
and, because $\sigma^{2M}$ commutes with $c$,
\[
 B_j^{\sigma^{2M}}=B_j^{-1}.
\]
Thus $T_j$ is fixed by $c$ and $\sigma^{2M}$.  Conversely, if a rotation
fixes $T_j$, then it sends $B_j$ to one of the two roots $B_j,B_j^{-1}$ of
$X^2-T_jX+1$.  The primitive orbit of $B_j$ forces that rotation to be
respectively $1$ or $\sigma^{2M}$.  Hence
\[
 \operatorname{Stab}(T_j)=\langle c,\sigma^{2M}\rangle,
\]
proving \eqref{eq:powered-trace-field-v12}.

Finally, because the class group $C_{4M}$ is cyclic of even order,
$\Cl(K)/\Cl(K)^2\simeq C_2$.  Genus theory therefore gives a unique
nontrivial genus character, whose kernel is $\langle\sigma^2\rangle$.
Since $D=-5n$, the corresponding real genus field is $\QQ(\sqrt5)$,
fixed in $H$ by $\langle c,\sigma^2\rangle$.  Because $2M$ is even,
\[
 \langle c,\sigma^{2M}\rangle
 \subseteq\langle c,\sigma^2\rangle,
\]
so $\QQ(\sqrt5)\subseteq\QQ(T_j)$ and the final degree is $h/4=M$.
\end{proof}

\begin{corollary}[The sixth-power reciprocal trace]
\label{cor:sixth-trace-v13}
Put
\[
 U_j=R_j^6,
 \qquad
 S_j=U_j+U_j^{-1}=R_j^6+R_j^{-6}.
\]
Then
\begin{equation}\label{eq:sixth-trace-field-v13}
 \QQ(S_j)=H^{\langle c,\sigma^{2M}\rangle}
 =\QQ(R_j^{24}+R_j^{-24}).
\end{equation}
Consequently
\[
 [\QQ(S_j):\QQ]=h/2,
 \qquad
 [\QQ(\sqrt5,S_j):\QQ(\sqrt5)]=h/4.
\]
\end{corollary}

\begin{proof}
Theorem~\ref{thm:primitive-powered-mirror-v12} gives
$\QQ(U_j)=H^{\langle c\rangle}$ and
$B_j=U_j^4$ with $B_j^{\sigma^{2M}}=B_j^{-1}$.  Therefore
\[
 (U_j^{\sigma^{2M}})^4=U_j^{-4}.
\]
The half-turn commutes with $c$, and
\[
 U_j=\frac{Z_j\overline{Z_j}}{125},
 \qquad Z_j=\mathfrak w_5(\tau_j)^6.
\]
Hence, in the distinguished complex embedding,
\[
 U_j^{\sigma^{2M}}
 =\frac{Z_j^{\sigma^{2M}}
        \overline{Z_j^{\sigma^{2M}}}}{125}>0.
\]
Since $U_j^{-1}>0$, uniqueness of the positive fourth root gives
\begin{equation}\label{eq:U-halfturn-v13}
 U_j^{\sigma^{2M}}=U_j^{-1}.
\end{equation}
Now a rotation fixing $S_j$ must send the primitive element $U_j$ to either
$U_j$ or $U_j^{-1}$, hence must be respectively $1$ or
$\sigma^{2M}$.  The reflection argument is identical to that in
Theorem~\ref{thm:primitive-powered-mirror-v12}; therefore
$\operatorname{Stab}(S_j)=\langle c,\sigma^{2M}\rangle$.  This proves the
first equality in \eqref{eq:sixth-trace-field-v13}.  The second follows
because the field generated by $R_j^{24}+R_j^{-24}$ was shown in the theorem
to have the same fixed group, and the degree statements follow as before.
\end{proof}

\begin{theorem}[Square-powered mirrors and the exact residual obstruction]
\label{thm:square-mirror-v13}
Under the standing hypotheses of Theorem~\ref{thm:all-mirrors-w5}, one has
\begin{equation}\label{eq:square-generator-v13}
 \QQ(R_j^2)=H^{\langle c\rangle}.
\end{equation}
Moreover, with $Y_j=R_j+R_j^{-1}$,
\begin{equation}\label{eq:squared-trace-v13}
 \QQ(Y_j^2)
 =H^{\langle c,\sigma^{2M}\rangle},
\end{equation}
so that
\begin{align}
 [\QQ(R_j^2):\QQ]&=h,\label{eq:R2-degree-v13}\\
 [\QQ(Y_j^2):\QQ]&=h/2,\label{eq:Y2-degree-v13}\\
 [\QQ(\sqrt5,Y_j^2):\QQ(\sqrt5)]&=h/4.
 \label{eq:Y2-genus-degree-v13}
\end{align}
Hence the unpowered degrees satisfy the exact dichotomies
\begin{equation}\label{eq:degree-dichotomy-v13}
 [\QQ(R_j):\QQ]\in\{h,2h\},
 \qquad
 [\QQ(Y_j):\QQ]\in\{h/2,h\}.
\end{equation}
Over the genus field,
\begin{equation}\label{eq:genus-dichotomy-v13}
 [\QQ(\sqrt5,Y_j):\QQ(\sqrt5)]
 \in\{h/4,h/2\}.
\end{equation}
\end{theorem}

\begin{proof}
By Lemma~\ref{lem:shape-D}, one has $3\nmid D$.
For $N=5$, the canonical generalized-Weber exponent is $6$.
Since $3\nmid D$, Enge--Morain's lower-power analysis
\cite[\S6.1.1]{EngeMorain} lowers this to exponent $e=2$.  More precisely,
according as $D\equiv1$ or $2\pmod3$, their conditions $1{:}3$ or $2{:}3$
apply.  In the notation of that section these conditions force $3\mid B$;
while $5\mid D$ and the level-$5$ condition force $5\mid B$.  Thus one may
choose the resulting $15$-system with
\[
 \gcd(A,15)=1,\qquad 15\mid B,\qquad 5\mid C,
\]
and for every class its singular value
\[
 W_j:=\mathfrak w_5(\tau_j)^2\in H.
\]
The representative still satisfies the level-$5$ condition used in
Theorem~\ref{thm:all-mirrors-w5}, so
\[
 R_j^2
 =\frac{W_j\overline{W_j}}5
 \in H^{\langle c\rangle}.
\]
Put $V_j=R_j^2$.  Since $V_j^3=R_j^6=U_j$, Corollary~\ref{cor:sixth-trace-v13}
and Theorem~\ref{thm:primitive-powered-mirror-v12} give
\[
 H^{\langle c\rangle}=\QQ(U_j)
 \subseteq\QQ(V_j)
 \subseteq H^{\langle c\rangle},
\]
which proves \eqref{eq:square-generator-v13}.

The half-turn again commutes with $c$, so
\[
 V_j^{\sigma^{2M}}
 =\frac{W_j^{\sigma^{2M}}
        \overline{W_j^{\sigma^{2M}}}}5>0.
\]
On the other hand
\[
 (V_j^{\sigma^{2M}})^3
 =U_j^{\sigma^{2M}}
 =U_j^{-1}=V_j^{-3}
\]
by \eqref{eq:U-halfturn-v13}.  Positivity therefore gives
\begin{equation}\label{eq:V-halfturn-v13}
 V_j^{\sigma^{2M}}=V_j^{-1}.
\end{equation}
Since $V_j$ generates $H^{\langle c\rangle}$, the same stabilizer argument
shows
\[
 \QQ(V_j+V_j^{-1})
 =H^{\langle c,\sigma^{2M}\rangle}.
\]
But
\[
 Y_j^2=(R_j+R_j^{-1})^2
      =V_j+V_j^{-1}+2,
\]
so adjoining $Y_j^2$ gives the same field, proving
\eqref{eq:squared-trace-v13} and the degree formulas.

Finally $R_j$ is a square root of the primitive element $V_j$, while $Y_j$
is a square root of $Y_j^2$.  Thus the corresponding extensions have degree
at most two, and the displayed dichotomies follow.
\end{proof}

\begin{proof}[Proof of Theorem~\ref{thm:main}]

Retain the lower-power $15$-system used in the proof of
Theorem~\ref{thm:square-mirror-v13}.  Thus a representative of $g^j$ is
$[A,B,C]$ with $\gcd(A,15)=1$, $5\mid C$ and $15\mid B$, and
\[
 \tau=\frac{-B+\sqrt D}{2A},
 \qquad
 x=\mathfrak w_5(\tau)^2\in H.
\]
The condition $15\mid B$ is exactly what makes the inverse form compatible
with the same $15$-system: $[A,-B,C]$ has the same middle coefficient modulo
$30$.  Its CM point is $-\overline\tau$ and represents $g^{-j}$.

Write $W_5z=-5/z$ for the Fricke involution.  Dedekind's eta inversion
formula gives the exact identity
\begin{equation}\label{eq:w5-Fricke-v14}
 \mathfrak w_5(W_5z)
 =\frac{\sqrt5}{\mathfrak w_5(z)}.
\end{equation}
Indeed,
$\eta(-1/z)=\sqrt{-iz}\,\eta(z)$ applied to $z$ and $z/5$ gives
\eqref{eq:w5-Fricke-v14}.  Since the $q$-expansion of $\eta$ is real,
$\mathfrak w_5(-\overline z)=\overline{\mathfrak w_5(z)}$.
Therefore
\begin{equation}\label{eq:R-ray-quotient-v14}
 R_j=
 \frac{\mathfrak w_5(\tau)}
      {\mathfrak w_5(W_5(-\overline\tau))}.
\end{equation}
To check the class attached to the denominator, note that
$W_5(-\overline\tau)$ is a root of
\[
 [C/5,B,5A].
\]
This form represents $g^{2M-j}$.  It is again an admissible member of the
same $15$-system: because $D$ is fundamental, $5\mid D$ and $5\nmid A$, the
identity $B^2-D=4AC$ gives $5\Vert C$; and because $3\nmid D$ while
$3\mid B$, one has $3\nmid C$.  Hence $\gcd(C/5,15)=1$, while its middle
coefficient is still $B$ modulo $30$.

We now descend the quotient in \eqref{eq:R-ray-quotient-v14}.  For
$N=5$ and exponent $e=1$, Enge--Morain's Theorem~3.2 gives
$\mathfrak w_5\in\mathcal F_{30}$; Shimura reciprocity
\cite[Theorem~2.2]{EngeMorain} therefore places all singular values appearing
above in the ray class field $L$ of modulus $30$ over $K$.  The Hilbert class
field $H$ is contained in $L$, and $\Gal(L/K)$ is abelian.  Put
\[
 a=\mathfrak w_5(\tau),
 \qquad
 b=\mathfrak w_5(W_5(-\overline\tau)).
\]
Their squares $a^2,b^2$ lie in $H$ by the lower-power theorem and correspond
to the classes $g^j$ and $g^{2M-j}$.  Hence some
$\widetilde\sigma\in\Gal(L/K)$ whose restriction to $H$ is
$\sigma^{2M-2j}$ satisfies
\[
 \widetilde\sigma(a)=\delta b,
 \qquad \delta\in\{\pm1\}.
\]
For $\rho\in\Gal(L/H)$, the fact that $a^2\in H$ gives
$\rho(a)=\chi(\rho)a$ with $\chi(\rho)\in\{\pm1\}$.  Since $L/K$ is
abelian,
\[
 \rho(b)
 =\delta\widetilde\sigma(\rho(a))
 =\chi(\rho)b.
\]
Thus every element of $\Gal(L/H)$ fixes $a/b$.  By
\eqref{eq:R-ray-quotient-v14}, $R_j=a/b\in H$.  The same formula also shows
that $R_j$ is the positive real number
$|\mathfrak w_5(\tau)|^2/\sqrt5$, so $c(R_j)=R_j$ and consequently
\[
 R_j\in H^{\langle c\rangle}.
\]
Together with Theorem~\ref{thm:square-mirror-v13}, which already gives
$\QQ(R_j^2)=H^{\langle c\rangle}$, this proves
\eqref{eq:unpowered-generator-v14}.

It remains to determine the sign of the half-turn.  The point $W_5\tau$
is a root of $[C/5,-B,5A]$, which is likewise an admissible member of the
$15$-system and represents $g^{j+2M}$.  Hence, if
$\widetilde\tau\in\Gal(L/K)$ is any lift of $\sigma^{2M}$, then
\eqref{eq:w5-Fricke-v14} gives
\[
 (\widetilde\tau(a))^2
 =\mathfrak w_5(W_5\tau)^2=\frac5{a^2},
\]
and therefore
\[
 \widetilde\tau(a)=\varepsilon\frac{\sqrt5}{a},
 \qquad \varepsilon\in\{\pm1\}.
\]
The class difference from $a^2$ to $b^2$ is
$2M-2j=2(M-j)$, an even power of $\sigma$.  The unique genus character has
kernel $\langle\sigma^2\rangle$, so every lift of
$\sigma^{2M-2j}$ fixes $\sqrt5$.  Using the same
$\widetilde\sigma$ as above and the commutativity of $\Gal(L/K)$, we obtain
\[
 \widetilde\tau(b)
 =\delta\widetilde\sigma(\widetilde\tau(a))
 =\varepsilon\frac{\sqrt5}{b}.
\]
Since $a/b=R_j\in H$, its action is independent of the chosen lift, and
therefore
\[
 R_j^{\sigma^{2M}}
 =\frac{\widetilde\tau(a)}{\widetilde\tau(b)}
 =\frac{b}{a}=R_j^{-1},
\]
which proves \eqref{eq:R-halfturn-v14}.

Now $Y_j$ is fixed by $c$ and by $\sigma^{2M}$.  Conversely, if a rotation
$\sigma^s$ fixes $Y_j$, then it sends the primitive element $R_j$ to one of
the two roots $R_j,R_j^{-1}$ of $X^2-Y_jX+1$.  Since $R_j^2$ already has
trivial rotational stabilizer, the first possibility forces $s=0$ and the
second, by \eqref{eq:R-halfturn-v14}, forces $s=2M$.  Reflections add only
$c$.  Hence
\[
 \operatorname{Stab}(Y_j)=\langle c,\sigma^{2M}\rangle,
\]
which proves \eqref{eq:unpowered-trace-v14}.  The first two degrees in
\eqref{eq:exact-degrees-v14} follow immediately.  Finally the real genus
field is $\QQ(\sqrt5)$, fixed by $\langle c,\sigma^2\rangle$, and
$\langle c,\sigma^{2M}\rangle\subseteq\langle c,\sigma^2\rangle$; hence the
relative degree over $\QQ(\sqrt5)$ is $h/4=M$.
\end{proof}

\begin{remark}[what the descent uses]
The final square root is not obtained by a positivity argument inside
$H^{\langle c\rangle}$, which would be insufficient because that field need
not be totally real.  The essential extra input is instead ray-class
cancellation: the two Weber square roots in \eqref{eq:R-ray-quotient-v14}
transform by the same quadratic sign under $\Gal(L/H)$ because the ambient
ray class extension is abelian over $K$.  Their quotient therefore descends
to $H$ before positivity is used to identify it with the distinguished
mirror ratio.
\end{remark}

\section{The cyclic class-number-4 case}\label{sec:C4}

Suppose
\[
\Cl(K)\simeq C_4=\langle g\rangle,
\qquad
\mathfrak p=g^2.
\]
There is only one independent odd mode, $\Lambda_1$, and the transform
reduces to
\[
d_0=\Lambda_1.
\]
Hence
\[
C_D=5\sqrt5\,e^{-3\Lambda_1}.
\]

If $\chi(g)=i$, then
\[
\Lambda_1=
\frac{\sqrt{|D|}}{2\pi}L_K(1,\chi).
\]
Therefore
\[
L_K(1,\chi)
=
-\frac{2\pi}{3\sqrt{|D|}}
\log\frac{C_D}{5\sqrt5}.
\]

For $D=-155$, an exact CM class-polynomial calculation on
$X_0(5)$ gives the Fricke trace
\[
C+\frac{125}{C}=1250+560\sqrt5.
\]
Selecting the small positive root gives precisely
\[
C
=
625+280\sqrt5
-
50\sqrt{313+140\sqrt5},
\]
which is the value announced in \eqref{eq:intro-C155}.  The exact
$X_0(5)$ class-polynomial calculation is given in
Appendix~\ref{app:D155}.  Hence
\[
\cosh(3\Lambda_1)=56+25\sqrt5.
\]
Thus
\[
L_{\QQ(\sqrt{-155})}(1,\chi)
=
\frac{2\pi}{3\sqrt{155}}
\operatorname{arcosh}(56+25\sqrt5).
\]

Equivalently, if
\[
u=e^{3\Lambda_1},
\]
then
\[
u+u^{-1}=112+50\sqrt5
\]
and
\[
u^4-224u^3+46u^2-224u+1=0.
\]

\section{The cyclic class-number-8 case: D=-95}\label{sec:C8}

For
\[
D=-95,
\qquad
\Cl(K)=\langle g\rangle\simeq C_8,
\]
write $A_j=g^j$ for the ideal classes.  A convenient reduced-form
ordering is
\[
A_1=[2,1,12],\qquad
A_3=[3,1,8],\qquad
A_4=[5,5,6].
\]

There are two odd modes, $\Lambda_1$ and $\Lambda_3$.  Put
\[
u=
\left(\frac{F(A_4)}{F(A_0)}\right)^3,
\qquad
r=
\frac{F(A_3)}{F(A_1)}.
\]
Then the cosine transform gives
\[
\Lambda_1+\Lambda_3=\frac23\log u,
\]
and
\[
\Lambda_1-\Lambda_3=2\sqrt2\log r.
\]
Hence
\begin{align*}
\Lambda_1&=\frac13\log u+\sqrt2\log r,\\
\Lambda_3&=\frac13\log u-\sqrt2\log r.
\end{align*}

The norm-$5$ quantity satisfies
\[
u+u^{-1}
=
\frac{
21+10\sqrt5+
5\sqrt{31+14\sqrt5}
}{2}.
\]
Equivalently, the defining eta quotient satisfies $41<u<42$ (a direct
four-term product bound already suffices), so
\[
 u=\frac{A_{19}+\sqrt{A_{19}^2-4}}2,
 \qquad
C_{19}=\frac{5\sqrt5}{u}
=\frac{5\sqrt5}{2}
\left(A_{19}-\sqrt{A_{19}^2-4}\right),
\]
where $A_{19}=u+u^{-1}$.  Therefore
\begin{equation}
S(e^{-\pi/\sqrt{95}})
=\left(
\frac{11-C_{19}+\sqrt{(11-C_{19})^2+4}}2
\right)^{1/5}
\label{eq:D95-S-exact}
\end{equation}

The second eta ratio admits the exact evaluation
\[
r+r^{-1}
=
\frac{\sqrt5+\sqrt{2\sqrt5-1}}2,
\]
and therefore
\[
r^8+2r^6-5r^5+3r^4-5r^3+2r^2+1=0.
\]

Since $u+u^{-1}=A_{19}$ and $r+r^{-1}=B_{19}$, where
\[
 A_{19}=\frac{21+10\sqrt5+5\sqrt{31+14\sqrt5}}2,
 \qquad
 B_{19}=\frac{\sqrt5+\sqrt{2\sqrt5-1}}2,
\]
we may write the two values without auxiliary logarithmic roots:
\begin{align}
L_K(1,\chi_1)
&=\frac{2\pi}{\sqrt{95}}
\left(
\frac13\operatorname{arcosh}\!\frac{A_{19}}2
+\sqrt2\operatorname{arcosh}\!\frac{B_{19}}2
\right),\label{eq:D95-L1-arcosh}\\
L_K(1,\chi_3)
&=\frac{2\pi}{\sqrt{95}}
\left(
\frac13\operatorname{arcosh}\!\frac{A_{19}}2
-\sqrt2\operatorname{arcosh}\!\frac{B_{19}}2
\right).\label{eq:D95-L3-arcosh}
\end{align}
Numerically,
\[
L_K(1,\chi_1)=1.002863649022912222\ldots,
\qquad
L_K(1,\chi_3)=0.597148062169489849\ldots.
\]

\subsection{Companion ordinary Rogers--Ramanujan values}

The exact alternating evaluations also give ordinary Rogers--Ramanujan
values at the same arguments without another class-polynomial computation.
Let
\[
 \mathcal M(X,Y)=X^6Y^5-X^5Y^6+X^5Y-5X^4Y^2+10X^3Y^3
 -5X^2Y^4+XY^5+X-Y.
\]
Put $u=R(q)$, $s=S(q)=-R(-q)$ and $v=R(q^2)$.  Ramanujan's classical
degree-$2$ modular equation and the mixed equation recorded in
\cite[Theorems~5.2 and~5.7]{BerndtRRCF} are
\[
 uv^2=\frac{v-u^2}{v+u^2},
 \qquad
 usv^2=\frac{us-v}{s-u}.
\]
Eliminating $v$ gives
\begin{equation}\label{eq:RS-direct-modular}
 \mathcal M(R(q),S(q))=0,\qquad S(q)=-R(-q).
\end{equation}
Thus the radical values \eqref{eq:intro-S155} and
\eqref{eq:intro-S95} immediately give exact algebraic characterizations of
the ordinary continued fraction at the same CM arguments.  In particular,
\begin{align*}
 \mathcal M\!\left(R(e^{-\pi/\sqrt{95}}),
 S(e^{-\pi/\sqrt{95}})\right)&=0,\\
 R(e^{-\pi/\sqrt{95}})&=0.618033988718136289\ldots,
\end{align*}
and
\begin{align*}
 \mathcal M\!\left(R(e^{-\pi/\sqrt{155}}),
 S(e^{-\pi/\sqrt{155}})\right)&=0,\\
 R(e^{-\pi/\sqrt{155}})&=0.618033988749859251\ldots.
\end{align*}
The point of \eqref{eq:RS-direct-modular} is that an exact alternating
singular value from the cyclic Kronecker method automatically comes with an
ordinary Rogers--Ramanujan companion.

\subsection{Eta-duplication proof of the second ratio}

Set
\[
s=\sqrt{95},
\qquad
\sigma=\frac{-5+is}{10}.
\]
Using the classes
\[
R_0=[1,1,24],\quad
R_1=[2,1,12],\quad
R_3=[3,1,8],\quad
R_4=[5,5,6],
\]
the eta duplication formula and the transformation
\[
|\eta(-1/\tau)|=|\tau|^{1/2}|\eta(\tau)|
\]
give
\[
\frac{G_{95}}{G_{19/5}}
=
\frac{F(R_4)F(R_1)}
     {F(R_0)F(R_3)}.
\]
Cubing,
\[
r^3=
\frac{u}
{
(G_{95}/G_{19/5})^3
}.
\]
Naika's Theorem~3.1 evaluates $G_{95}$ and $G_{19/5}$ using a
degree-$19$ modular equation~\cite{Naika}.  In particular, the
auxiliary quantity
\[
v=\{16\alpha\beta(1-\alpha)(1-\beta)\}^{1/24}
\]
satisfies
\[
4v^4+4\sqrt2\,v^3+4v^2+\sqrt2\,v-1=0,
\]
with the relevant positive root
\[
v=\frac{-1+\sqrt{2\sqrt5-1}}{2\sqrt2}.
\]
Substitution into Naika's class-invariant relation, together with the
preceding eta-duplication identity, yields
\[
r+r^{-1}
=\frac{\sqrt5+\sqrt{2\sqrt5-1}}2.
\]
Thus the reciprocal octic displayed above follows algebraically and
does not rely on numerical recognition.

\section{The cyclic class-number-16 case: D=-895}
\label{sec:D895-certificate}

Let
\[
D=-895=-5\cdot179.
\]
The reduced form
\[
g=[2,1,112]
\]
has order $16$, so
\[
\Cl(-895)\simeq C_{16},
\qquad
\mathfrak p=g^8=[5,5,46].
\]
Define
\[
R_0=\frac{F(g^8)}{F(1)},\quad
R_1=\frac{F(g^7)}{F(g)},\quad
R_2=\frac{F(g^6)}{F(g^2)},\quad
R_3=\frac{F(g^5)}{F(g^3)}.
\]
Numerically,
\begin{align*}
R_0&=235.329995174721387634091727\ldots,\\
R_1&=10.258790189890517162985542\ldots,\\
R_2&=2.169677396369514425117150\ldots,\\
R_3&=0.929524693702105431173609\ldots.
\end{align*}

\subsection{One level-5 system for all four mirrors}

The general level-$5$ realization above becomes especially concrete here.
The following $5$-system is ordered by the powers $g^m$:
\[
\begin{array}{c|r|r|r@{\qquad}c|r|r|r}
 m&A_m&B_m&C_m&m&A_m&B_m&C_m\\ \hline
0&1&-5&230&8&46&-5&5\\
1&2&5&115&9&23&-5&10\\
2&4&-15&70&10&14&15&20\\
3&8&-15&35&11&7&15&40\\
4&14&55&70&12&14&-55&70\\
5&7&-15&40&13&8&15&35\\
6&14&-15&20&14&4&15&70\\
7&23&5&10&15&2&-5&115
\end{array}
\]
Every row satisfies $(A_m,5)=1$, $5\mid C_m$, and
$B_m\equiv5\pmod{10}$.  If
\[
\tau_m=\frac{-B_m+\sqrt{-895}}{2A_m},
\qquad
\mathfrak w_5(\tau)=\frac{\eta(\tau/5)}{\eta(\tau)},
\]
then $\tau_m/5$ is represented by $[5A_m,B_m,C_m/5]$, which reduces to
$g^{m+8}$.  Therefore, for $0\le m\le3$,
\begin{equation}
R_m=\frac1{\sqrt5}\,|\mathfrak w_5(\tau_m)|^2.
\label{eq:D895-all-w5}
\end{equation}

Put
\[
Z_m=\mathfrak w_5(\tau_m)^6.
\]
For $N=5$, the canonical exponent in Enge--Morain's theorem is $6$;
hence the $Z_m$ form the complete set of ring-class conjugates
\cite[Theorem~4.3]{EngeMorain}.  Since complex conjugation acts dihedrally,
if
\[
U_m=R_m^6=\frac{Z_m\overline{Z_m}}{125},
\]
then the sixteen conjugates of $U_m$ are
\begin{equation}
U_{m,i}=\frac{Z_{m+i}\,\overline{Z_{m-i}}}{125},
\qquad 0\le i<16,
\label{eq:D895-U-conjugates}
\end{equation}
with indices modulo $16$.

\subsection{Certified mirror polynomials}

For compactness write
\[
P_m(X)=X^{16}+a_{m,1}X^{15}+\cdots+a_{m,7}X^9+a_{m,8}X^8
+a_{m,7}X^7+\cdots+a_{m,1}X+1.
\]
The exact coefficients are
\[
\begin{array}{c|rrrrrrrr}
 m&a_{m,1}&a_{m,2}&a_{m,3}&a_{m,4}&a_{m,5}&a_{m,6}&a_{m,7}&a_{m,8}\\ \hline
0&-237&498&-24888&42088&1152&95508&-498078&767733\\
1&7&335&-5307&3689&-41217&115314&-259744&323219\\
2&39&524&3033&9488&11772&-21138&-107913&-154087\\
3&11&132&311&-8&-315&2565&-5442&5409
\end{array}
\]

\begin{theorem}[All four $D=-895$ mirrors]
\label{thm:D895-four-mirrors}
For $m=0,1,2,3$, the polynomial $P_m$ is irreducible over $\QQ$ and
\[
P_m(R_m)=0.
\]
Thus every mirror has degree $16$ over $\QQ$.
\end{theorem}

\begin{proof}
We give a reproducible certificate.  For
$q=e^{2\pi i\tau}$ put
\[
\eta_M(\tau)=e^{\pi i\tau/12}\prod_{n=1}^{M}(1-q^n),
\qquad \rho=|q|.
\]
The omitted logarithmic tail satisfies
\begin{equation}
\left|\sum_{n>M}\log(1-q^n)\right|
\le
\frac{\rho^{M+1}}{(1-\rho)(1-\rho^{M+1})}.
\label{eq:D895-w5-tail}
\end{equation}
The companion script \texttt{d895\_all\_mirrors\_certificate\_v8.py}
uses complex interval arithmetic to evaluate all sixteen $Z_m$.
The largest truncation length is only $1053$.

For each $m=0,1,2,3$ it forms
\[
Q_m(U)=\prod_{i=0}^{15}(U-U_{m,i}).
\]
The dihedral Galois action shows that the coefficients of $Q_m$ are
rational.  On the other hand, the integral modular equation for the canonical
power $\mathfrak w_5^6$ implies that every $Z_i$ is an algebraic integer
\cite[Theorem~7.3]{EngeMorain}.  Hence the coefficient of $U^{16-k}$ in
$Q_m$ lies in $125^{-k}\mathbb Z$.  The smallest denominator spacing that can
occur is therefore at least $125^{-16}$.  Every computed coefficient interval
contains the displayed integer and has radius below $10^{-109}$, far smaller
than $\tfrac12 125^{-16}$; this certifies the coefficient exactly.  Exact
polynomial arithmetic then verifies
\[
P_m(X)\mid Q_m(X^6)
\qquad (m=0,1,2,3).
\]
The interval for the positive number $R_m$ excludes zero from the complementary
factor, so $P_m(R_m)=0$.  Exact factorization over $\mathbb{Z}$ proves each $P_m$
irreducible.
\end{proof}

\subsection{Compact quartics over \texorpdfstring{$\QQ(\sqrt5)$}{Q(sqrt(5))}}

The reciprocal degree-$16$ equations admit a much smaller reader-facing
form.  Put
\[
Y_m=R_m+R_m^{-1}.
\]
Then the relevant quartic factor over $\QQ(\sqrt5)$ is as follows:
\begin{align}
0={}&Y_0^4-\frac{237+103\sqrt5}{2}Y_0^3
-\frac{291+225\sqrt5}{2}Y_0^2
+\frac{225+409\sqrt5}{2}Y_0
+1288+447\sqrt5,
\label{eq:D895-q0}\\[1mm]
0={}&Y_1^4+\frac{7-11\sqrt5}{2}Y_1^3
+(233-110\sqrt5)Y_1^2
+\frac{-937+377\sqrt5}{2}Y_1
+411-118\sqrt5,
\label{eq:D895-q1}\\[1mm]
0={}&Y_2^4+\frac{39+9\sqrt5}{2}Y_2^3
+\frac{237+69\sqrt5}{2}Y_2^2
-\frac{309+225\sqrt5}{2}Y_2
-382-219\sqrt5,
\label{eq:D895-q2}\\[1mm]
0={}&Y_3^4+\frac{11-\sqrt5}{2}Y_3^3
+\frac{95-43\sqrt5}{2}Y_3^2
+\frac{-181+59\sqrt5}{2}Y_3
+\frac{123-55\sqrt5}{2}.
\label{eq:D895-q3}
\end{align}
The certified isolating intervals are
\begin{align*}
235.3342445267&<Y_0<235.3342445268,\\
10.3562675708&<Y_1<10.3562675710,\\
2.6305754089&<Y_2<2.6305754091,\\
2.0053433424&<Y_3<2.0053433426.
\end{align*}
Consequently
\begin{equation}
R_m=
\begin{cases}
\dfrac{Y_m+\sqrt{Y_m^2-4}}2,&m=0,1,2,\\[2mm]
\dfrac{Y_3-\sqrt{Y_3^2-4}}2,&m=3.
\end{cases}
\label{eq:D895-R-from-Y}
\end{equation}
Thus all four mirrors are specified by quartics over the same quadratic
field $\QQ(\sqrt5)$.
\begin{remark}[Independent level-$4$ realization of $R_3$]
The fourth mirror also has the identity
\[
R_3=\frac12\left|\frac{\eta(\tau_3/4)}{\eta(\tau_3)}\right|^2,
\qquad
\tau_3=\frac{-1+\sqrt{-895}}{16}.
\]
The level-$4$ generalized-Weber calculation from the previous version of
these notes gives the same polynomial $P_3$ independently.  The unified
level-$5$ certificate above is preferable in the main argument because it
handles all four mirrors simultaneously.
\end{remark}

\subsection{All four odd class-character values}

Put $d_m=\log R_m$.  The inverse cosine transform now gives the four odd
modes explicitly:
\begin{align}
\Lambda_1={}&d_0+\sqrt{2+\sqrt2}\,d_1+\sqrt2\,d_2
+\sqrt{2-\sqrt2}\,d_3,\\
\Lambda_3={}&d_0+\sqrt{2-\sqrt2}\,d_1-\sqrt2\,d_2
-\sqrt{2+\sqrt2}\,d_3,\\
\Lambda_5={}&d_0-\sqrt{2-\sqrt2}\,d_1-\sqrt2\,d_2
+\sqrt{2+\sqrt2}\,d_3,\\
\Lambda_7={}&d_0-\sqrt{2+\sqrt2}\,d_1+\sqrt2\,d_2
-\sqrt{2-\sqrt2}\,d_3.
\end{align}
Since
\[
L_K(1,\chi_k)=\frac{2\pi}{\sqrt{895}}\Lambda_k,
\]
we obtain
\begin{align*}
L_K(1,\chi_1)&=2.268740533538334759252363284\ldots,\\
L_K(1,\chi_3)&=1.319470763923675756120338522\ldots,\\
L_K(1,\chi_5)&=0.514275485895935414478872584\ldots,\\
L_K(1,\chi_7)&=0.485261960685057495469316063\ldots.
\end{align*}
For a more symmetric exact form, set
\[
 a_m=\operatorname{arcosh}(Y_m/2)\quad(m=0,1,2),
 \qquad
 a_3=-\operatorname{arcosh}(Y_3/2),
\]
where each $Y_m$ is the uniquely isolated quartic root in
\eqref{eq:D895-q0}--\eqref{eq:D895-q3}.  Then $a_m=\log R_m$, so, for
example,
\begin{equation}\label{eq:D895-L1-arcosh}
L_K(1,\chi_1)=\frac{2\pi}{\sqrt{895}}
\left(
 a_0+\sqrt{2+\sqrt2}\,a_1+\sqrt2\,a_2
 +\sqrt{2-\sqrt2}\,a_3
\right).
\end{equation}
The other three displayed cosine formulas give the analogous exact
$\operatorname{arcosh}$ evaluations.  Thus the $C_{16}$ example separates
four genuinely non-genus Hecke values using only quartic algebraic inputs
over $\QQ(\sqrt5)$.

\subsection{The alternating Rogers--Ramanujan value at D=-895}

The first mirror also gives an exact singular value.  Define
\[
C_{179}=S(e^{-\pi/\sqrt{895}})^{-5}+11-S(e^{-\pi/\sqrt{895}})^5.
\]
Ramanathan's formula and \eqref{eq:D895-R-from-Y} give
\[
C_{179}=\frac{5\sqrt5}{R_0^3}.
\]
Hence
\begin{equation}
S(e^{-\pi/\sqrt{895}})
=
\left(
\frac{11-C_{179}+\sqrt{(11-C_{179})^2+4}}2
\right)^{1/5}
\label{eq:D895-S}
\end{equation}

\section{Higher certified stages: \texorpdfstring{$C_{32}$ and $C_{64}$}{C32 and C64}}
\label{sec:higher-certified}

The low-degree cases above already illustrate the transform in closed form.
To show that the same mechanism remains computationally effective at much
larger class number, we record two independently certified stages.  These
examples are not needed for Theorem~\ref{thm:main}; their role is to provide
large exact class-polynomial certificates and explicit higher-order Hecke
values.

\subsection*{D=-3295: a certified \texorpdfstring{$C_{32}$}{C32} orbit}
Put
\[
 D=-3295=-5\cdot659,
 \qquad K=\QQ(\sqrt{-3295}).
\]
Exact Gauss composition gives
\[
 \Cl(K)=\langle g\rangle\simeq C_{32},
 \qquad g=[2,-1,412],
 \qquad g^{16}=[5,5,166].
\]
For
\[
 R_j=\frac{F(g^{16-j})}{F(g^j)},\qquad 0\le j<8,
\]
the companion interval certificate constructs the thirty-two conjugates of
$U_j=R_j^6$ from a single level-$5$ system and certifies the resulting
polynomial coefficientwise.  The largest eta truncation is $1670$, and the
worst coefficient enclosure error over all eight mirrors is less than
$1.2\cdot10^{-40}$.  Exact factorization then independently recovers the
degrees predicted by Theorem~\ref{thm:main}: $32$ for $R_j$, $16$ for
$Y_j=R_j+R_j^{-1}$, and $8$ over $\QQ(\sqrt5)$.

A representative genus-field equation is obtained from the last mirror.
If $Y=Y_7$, then
\begin{equation}
\begin{aligned}
0={}&Y^8+(32-22\sqrt5)Y^7+(4603-2012\sqrt5)Y^6\\
&+\frac{100093-44783\sqrt5}{2}Y^5
 +(217850-97589\sqrt5)Y^4\\
&+(445906-199556\sqrt5)Y^3
 +(418832-187201\sqrt5)Y^2\\
&+\frac{298331-133801\sqrt5}{2}Y
 +\frac{39603-17711\sqrt5}{2}.
\end{aligned}
\label{eq:D3295-Y7}
\end{equation}
All eight odd modes are then given exactly by
\begin{equation}
 \Lambda_k=\log R_0+2\sum_{j=1}^{7}\log R_j\cos\frac{k\pi j}{16},
 \qquad
 L_K(1,\chi_k)=\frac{2\pi}{\sqrt{3295}}\Lambda_k,
 \quad k\ \text{odd}.
\label{eq:D3295-Lvalues}
\end{equation}
For orientation,
\[
L_K(1,\chi_1)=2.926251025060943\ldots,
\qquad
L_K(1,\chi_{15})=0.367360614679352\ldots.
\]
The norm-$5$ mirror gives
\[
 C_{659}=\frac{5\sqrt5}{R_0^3},
\qquad
S(e^{-\pi/\sqrt{3295}})
=
\left(
\frac{11-C_{659}+\sqrt{(11-C_{659})^2+4}}2
\right)^{1/5}
\]

\subsection*{D=-8495: a certified \texorpdfstring{$C_{64}$}{C64} orbit}
Now put
\[
 D=-8495=-5\cdot1699,
 \qquad K=\QQ(\sqrt{-8495}).
\]
Exact composition gives
\[
 \Cl(K)=\langle g\rangle\simeq C_{64},
 \qquad g=[2,-1,1062],
 \qquad g^{32}=[5,5,426].
\]
For the sixteen mirrors
\[
 R_j=\frac{F(g^{32-j})}{F(g^j)},\qquad 0\le j<16,
\]
the interval certificate uses one $64$-element level-$5$ system.  Across all
sixteen norm polynomials the longest eta truncation is $5260$ and the worst
coefficient enclosure error is less than $1.9\cdot10^{-65}$.  Reciprocal
trace descent and exact factorization recover degree $64$ for $R_j$, degree
$32$ for $Y_j$, and degree $16$ over $\QQ(\sqrt5)$, again exactly as in
Theorem~\ref{thm:main}.

The inverse cosine transform separates all sixteen odd modes:
\begin{equation}
 \Lambda_k=\log R_0+2\sum_{j=1}^{15}\log R_j\cos\frac{k\pi j}{32},
 \qquad
 L_K(1,\chi_k)=\frac{2\pi}{\sqrt{8495}}\Lambda_k,
 \quad k\ \text{odd}.
\label{eq:D8495-Lvalues}
\end{equation}
For example,
\[
L_K(1,\chi_3)=4.281325086211408\ldots,
\qquad
L_K(1,\chi_{29})=0.206180862376204\ldots.
\]
The corresponding alternating value is
\[
 C_{1699}=\frac{5\sqrt5}{R_0^3},
\qquad
S(e^{-\pi/\sqrt{8495}})
=\left(
\frac{11-C_{1699}+\sqrt{(11-C_{1699})^2+4}}2
\right)^{1/5}
\]
with
\[
S(e^{-\pi/\sqrt{8495}})
=1.618033988749894848204586578\ldots.
\]

\begin{remark}[Certified hierarchy]
Together with the exact $C_8$ and certified $C_{16}$ cases, these examples
independently realize the sequence
\[
\begin{array}{c|c|c|c}
\Cl(K)&[\QQ(R_j):\QQ]&[\QQ(Y_j):\QQ]
&[\QQ(\sqrt5,Y_j):\QQ(\sqrt5)]\\ \hline
C_8&8&4&2\\
C_{16}&16&8&4\\
C_{32}&32&16&8\\
C_{64}&64&32&16.
\end{array}
\]
Theorem~\ref{thm:main} proves this degree pattern uniformly; the large
computations provide independent exact certificates and explicit singular
values rather than evidence on which the theorem depends.
\end{remark}

\appendix
\section{The level-5 class polynomial at D=-155}
\label{app:D155}

We record the short exact calculation used in the class-number-$4$ example.
The reduced primitive forms of discriminant $-155$ are
\[
 [1,1,39],\quad [3,-1,13],\quad [3,1,13],\quad [5,5,9],
\]
and the Hilbert class polynomial is
\begin{align*}
H_{-155}(J)={}&J^4+96905542950912000J^3
 -44477871096357453824000J^2\\
&+20396251654725321097216000000J\\
&+37425860028464856284790784000000.
\end{align*}
Choose the principal $5$-system representative
\[
 \tau=\frac{-5+\sqrt{-155}}2,
\qquad
 Z=\mathfrak w_5(\tau)^6
   =\left(\frac{\eta(\tau/5)}{\eta(\tau)}\right)^6.
\]
Then $\tau/5$ represents the ramified class $[5,5,9]$.  The standard
level-$5$ Hauptmodul relation~\cite[Table~5]{Maier} is
\begin{equation}\label{eq:D155-jZ}
 j(\tau/5)=\frac{(Z^2+250Z+3125)^3}{Z^5}.
\end{equation}
Substituting \eqref{eq:D155-jZ} into $H_{-155}$ and factoring over
$\QQ[Z]$ gives a unique degree-$4$ factor
\begin{equation}\label{eq:D155-P4}
 P_{155}(Z)=Z^4+2500Z^3-5250Z^2+312500Z+15625.
\end{equation}
The generalized-Weber class-invariant theorem shows that the four values of
$\mathfrak w_5^6$ form the complete class orbit, so the principal value is a
root of \eqref{eq:D155-P4}.  Over $\QQ(\sqrt5)$,
\begin{align*}
P_{155}(Z)
={}&\bigl(Z^2+(1250+560\sqrt5)Z+125\bigr)\\
&\times\bigl(Z^2+(1250-560\sqrt5)Z+125\bigr).
\end{align*}
To select the branch rigorously, put
$\rho=\exp(-\pi\sqrt{155}/5)<4.01\cdot10^{-4}$.  For the eta product,
\[
 \left|\sum_{n>N}\log(1-q^n)\right|
 \le
 \frac{\rho^{N+1}}{(1-\rho)(1-\rho^{N+1})}.
\]
With $N=4$ the right side is less than $1.04\cdot10^{-17}$; the first four
factors, enclosed with this tail, give $-2503<Z<-2502$.  Hence the first
quadratic is selected, and
\[
 Z+\frac{125}{Z}=-(1250+560\sqrt5).
\]
For the norm-$5$ mirror,
\[
 R_0=\frac1{\sqrt5}|\mathfrak w_5(\tau)|^2,
 \qquad
 R_0^3=-\frac{Z}{5\sqrt5},
\]
so Ramanathan's quantity $C_{31}=5\sqrt5/R_0^3$ is
$C_{31}=-125/Z$.  Consequently
\[
 C_{31}+\frac{125}{C_{31}}=1250+560\sqrt5.
\]
Selecting the small positive root yields
\[
 C_{31}=625+280\sqrt5-50\sqrt{313+140\sqrt5},
\]
which is the value used in Section~\ref{sec:C4}.

\section*{Declarations}

\paragraph{Author contributions.}
The sole author conceived the study, developed the proofs, performed the
computations, and wrote and revised the manuscript.

\paragraph{Funding.}
Not applicable.

\paragraph{Competing interests.}
The author declares no competing interests.

\paragraph{Data and code availability.}
No external datasets were used.  Exact class-polynomial and interval
certificates supporting the explicit high-class-number evaluations are
described in the text.  The corresponding scripts are retained separately
for independent verification.

\end{document}